\documentclass[12pt]{amsart}
\usepackage{amsfonts}
\usepackage{latexsym,amsmath,amsthm,amssymb,amsfonts,mathrsfs}
\usepackage{graphicx}
\numberwithin{equation}{section}

\newtheorem{thm}{Theorem}[section]

\newtheorem{lem}{Lemma}[section]

\newtheorem{cor}{Corollary}[section]
\theoremstyle{definition}
\newtheorem{dfn}{Definition}[section]

\theoremstyle{remark}
\newtheorem{rem}{Remark}

\newcommand{\ppt}[1]{\frac{\partial {#1}}{\partial t}}

\begin{document}

\title[Parabolic Omori-Yau Maximum Principles for MCF]{Parabolic Omori-Yau maximum principle for mean curvature flow and some applications}

\author{John Man Shun Ma}

\address{Department of Mathematics,
The University of British Columbia, Vancouver, BC 
Canada V6T1Z2}
\email{johnma@math.ubc.edu.ca}

\begin{abstract}
We derive a parabolic version of Omori-Yau maximum principle for a proper mean curvature flow when the ambient space has lower bound on $\ell$-sectional curvature. We apply this to show that the image of Gauss map is preserved under a proper mean curvature flow in euclidean spaces with uniform bounded second fundamental forms. This generalizes the result of Wang \cite{Wang} for compact immersions. We also prove a Omori-Yau maximum principle for properly immersed self-shrinkers, which improves a result in \cite{CJQ}.
\end{abstract}

\date{\today}

\maketitle
\section{Introduction}
Let $(M, g)$ be a Riemannian manifold and let $u: M \to \mathbb R$ be a twice differentiable function. If $M$ is compact, $u$ is maximized at some point $x\in M$. At this point, basic advanced calculus implies
$$ u(x) = \sup u, \ \ \ \nabla ^M u(x) = 0, \ \ \ \Delta^M u (x) \le 0.$$
Here $\nabla ^M$ and $\Delta^M$ are respectively the gradient and Laplace operator with respect to the metric $g$. When $M$ is noncompact, a bounded function might not attain a maximum. In this situation, Omori \cite{Omori} and later Yau \cite{Yau} provide some noncompact versions of maximum principles. We recall the statement in \cite{Yau}:
\begin{thm} \label{OY Yau version}
Let $(M,g)$ be a complete noncompact Riemannian manifold with bounded below Ricci curvature. Let $u: M\to \mathbb R$ be a bounded above twice differentiable function. Then there is a sequence $\{x_i\}$ in $M$ such that
$$u(x_i) \to \sup u, \ \ \ |\nabla u|(x_i) \to 0, \ \ \ \limsup_{i\to \infty} \Delta^M u (x_i) \le 0. $$
\end{thm}

Maximum principles of this form are called Omori-Yau maximum principles. The assumption on the lower bound on Ricci curvature in Theorem \ref{OY Yau version} has been weaken in (e.g.) \cite{CX}, \cite{PRS}. On the other hand, various Omori-Yau type maximum principles have been proved for other elliptic operators and on solitons in geometric flows, such as Ricci solition \cite{CJQ} and self-shrinkers in mean curvature flows \cite{CP}. The Omori-Yau maximum principles are powerful tools in studying noncompact manifolds and have a lot of geometric applications. We refer the reader to the book \cite{AMR} and the reference therein for more information.

In this paper, we derive the following parabolic version of Omori-Yau maximum principle for mean curvature flow.

\begin{thm}  [Parabolic Omori-Yau Maximum Principle]\label{parabolic OY}
Let $n\ge 2$ and $m\ge 1$. Let $(\overline M^{n+m}, \bar g)$ be a $n+m$-dimensional noncompact complete Riemannian manifold such that the $(n-1)$-sectional curvature of $\overline M$ is bounded below by $-C$ for some positive constant $C$.  Let $M^n$ be a $n$-dimensional noncompact manifold and let $F : M^n \times [0,T]\to \overline M$ be a proper mean curvature flow. Let $u: M\times [0,T] \to \mathbb R$ be a continuous function which satisfies
\begin{enumerate}
\item $\sup_{(x, t)\in M \times [0,T]} u > \sup_{x\in M} u(\cdot, 0)$,
\item $u$ is twice differentiable in $M \times (0,T]$, and
\item (sublinear growth condition) There are $B>0$, $\alpha\in [0,1)$ and some $y_0\in \overline M$ so that
\begin{equation}
u(x, t) \le B (1+d_{\overline M} (y_0, F(x, t))^\alpha),\ \ \ \forall (x, t) \in M\times [0,T].
\end{equation}
\end{enumerate}
Then there is a sequence of points $(x_i, t_i) \in M \times (0,T]$ so that
\begin{equation}
 u(x_i, t_i) \to \sup u,\ \ |\nabla^{M_{t_i}} u (x_i, t_i)| \to  0,\ \ \ \liminf_{i\to \infty}\left(\ppt{} - \Delta^{M_{t_i}}\right) u (x_i, t_i) \ge 0 .
\end{equation}
\end{thm}
We remark that the above theorem makes no assumption on the curvature of the immersion $F_t$. See section 2 for the definition of $\ell$-sectional curvature.

With this parabolic Omori-Yau maximum principle, we derive the following results.

In \cite{Wang}, the author studies the gauss map along the mean curvature flow in the euclidean space. He shows that if the image of the gauss map stays inside a geodesic submanifold in the Grassmanians, the same is also true along the flow when the initial immersion is compact. As a first application, we extend Wang's theorem to the noncompact situation.

\begin{thm} \label{main theorem}
Let $F_0: M^n \to \mathbb R^{n+m}$ be a proper immersion and let $F :M^n \times [0,T] \to \mathbb R^{n+m}$ be a mean curvature flow of $F_0$ with uniformly bounded second fundamental form. Let $\Sigma$ be a compact totally geodesic submanifold of the Grassmanians of $n$-planes in $\mathbb R^{n+m}$. If the image of the Gauss map $\gamma$ satisfies $\gamma (\cdot,0) \subset \Sigma$, then $\gamma (\cdot, t) \subset \Sigma$ for all $t\in [0,T]$.
\end{thm}
As a corollary, we have the following:
\begin{cor} \label{preserve Lagrangian}
Let $F_0: M^n \to \mathbb R^{2n}$ be a proper Lagrangian immersion and let $F :M\times [0,T] \to \mathbb R^{2n}$ be a mean curvature flow with uniformly bounded second fundamental form. Then $F_t$ is Lagrangian for all $t\in [0,T]$.
\end{cor}
The above result is well-known when $M$ is compact \cite{Sm}, \cite{Wang}. Various forms of Corollary \ref{preserve Lagrangian} are known to the experts (see remark \ref{Neves, Sm} below).

The second application is to derive a Omori-Yau maximum principle for the $\mathcal L$-operator of a proper self-shrinker. The $\mathcal L$ operator is introduced in \cite{CM} when the authors study the entropy stability of a self-shrinker. Since then it proves to be an important operator in mean curvature flow. Using Theorem \ref{parabolic OY}, we prove

\begin{thm} \label{OY for self shrinker}
Let $\widetilde F : M^n\to \mathbb R^{n+m}$ be a properly immersed self-shrinker and let $f: M^n\to \mathbb R$ be a twice differentiable function so that
\begin{equation}
f(x) \le C(1+ |\widetilde F(x) |^\alpha)
\end{equation}
for some $C>0$ and $\alpha \in [0,1)$. Then there exists a sequence $\{x_i\}$ in $M$ so that
\begin{equation}
f(x_i) \to \sup_M f,\ \ \ |\nabla f|(x_i)\to 0,\ \ \ \limsup_{i\to \infty} \mathcal L f(x_i) \le 0.
\end{equation}
\end{thm}

The above theorem is a generalization of Theorem 5 in \cite{CJQ} since we assume weaker conditions on $f$.

In section 2, we prove the parabolic Omori-Yau maximum principle. In section 3 we prove Theorem \ref{main theorem} and in section 4 we prove Theorem \ref{OY for self shrinker}. The author would like to thank Jingyi Chen for the discussion on Omori-Yau maximum principle and Kwok Kun Kwong for suggesting the work of Li and Wang \cite{LW}. 


\section{Proof of the parabolic Omori-Yau maximum principle}
Let $(\overline M^{n+m}, \overline g)$ be an $n+m$ dimensional complete noncompact Riemannian manifold. Let $F : M \times [0,T]\to \overline M$, where $M$ is an $n$-dimensional noncompact manifold, be a family of immersions $\{ F(\cdot, t): M \to \overline M\}$ which satisfies the mean curvature flow equation
\begin{equation} \label{MCF eqn}
\frac{\partial F}{\partial t } (x, t)= \vec H (x, t).
\end{equation}
Here $\vec H (x, t)$ is the mean curvature vector given by
\begin{equation}
\vec H = \text{tr} A
\end{equation}
and $A(X, Y) = (\overline \nabla_X Y)^\perp$ is the second fundamental form of the immersion $F(\cdot, t)$.

Next we recall the definition of $\ell$-sectional curvature in \cite{LW}. Let $\overline M^N$ be an $N$-dimensional Riemannian manifold. Let $p \in \overline M$, $1\le \ell \le N-1$. Consider a pair $\{ w, V\}$, where $w\in T_p\overline M$ and $V\subset T_p\overline M$ is a $\ell$-dimensional subspace so that $w$ is perpendicular to $V$. 

\begin{dfn} 
The $\ell$-sectional curvature of $\{w, V\}$ is given by
\begin{equation} \label{dfn of l section cur}
K^\ell_{\overline M}(w, V) = \sum_{i=1}^\ell \langle R(w, e_i)w, e_i\rangle,
\end{equation}
where $R$ is the Riemann Curvature tensor on $\overline M$ and $\{e_1, \cdots, e_\ell\}$ is any orthonormal basis of $V$. 
\end{dfn}

We say that $\overline M$ has $\ell$-sectional curvature bounded from below by a constant $C$ if 
\begin{equation*}
K^\ell_{\overline M} (w, V) \ge \ell C
\end{equation*}
for all pairs $\{w, V\}$ at any point $p\in M$. In \cite{LW}, the authors prove the following comparison theorem for the distance function $r$ on manifolds with lower bound on $\ell$-sectional curvatures. 

\begin{thm} \label{comp thm of LW} [Theorem 1.2 in \cite{LW}]
Assume that $\overline M$ has $\ell$-sectional curvature bounded from below by $-C$ for some $C>0$. Let $p \in M$ and $r(x) = d_{\overline g} (x, p)$. If $x$ is not in the cut locus of $p$ and $V\subset T_x \overline M$ is perpendicular to $\nabla r(x)$, then
\begin{equation}
\sum_{i=1}^\ell \nabla^2 r(e_i, e_i)\le \ell \sqrt C \coth (\sqrt C r),
\end{equation}
where $\{e_1,\cdots , e_{\ell}\}$ is any orthonormal basis of $V$.
\end{thm}

Now we prove Theorem \ref{parabolic OY}. We recall that $F$ is assumed to be proper, and $u$ satisfies condition (1)-(3) in the statement of Theorem \ref{parabolic OY}.

\begin{proof} [Proof of Theorem \ref{parabolic OY}]
Adding a constant to $u$ if necessary, we assume
$$\sup_{x\in M} u(x, 0) = 0.$$
By condition (1) in Theorem \ref{parabolic OY}, we have $u(y, s) >0$ for some $(y, s)$. Note that $s>0$. Let $y_0 \in \overline M$ and $r (y) = d_{\bar g} (y, y_0)$ be the distance to $y_0$ in $\overline M$. Let $\rho(x, t) = r(F(x, t))$. Note that $u(y,s) - \epsilon \rho(y, s)^2 >0$ whenever $\epsilon$ is small. Let $(\bar x_i, s_i)$ be a sequence so that $ u(\bar x_i, s_i) \to \sup u \in (0, \infty]$. Let $\{\epsilon_i\}$ be a sequence in $(0,\epsilon)$ converging to $0$ which satisfies
\begin{equation} \label{choice of epsilon i}
 \epsilon_i \rho(\bar x_i, s_i)^2 \le \frac {1}{i}, \ \ \ i=1, 2, \cdots.
\end{equation}
Define
\[ u_i (x, t) = u(x, t) - \epsilon_i \rho(x,t)^2. \]
Note that $u_i (y, s) >0$ and $u_i (\cdot, 0) \le 0$. Using condition (3) in Theorem \ref{parabolic OY}, there is $R>0$ so that $u_i (x, t)\le 0$ when $F(x, t) \notin B_R(y_0)$, the closed ball in $\overline M$ centered at $y_0$ with radius $R$. Since $\overline M$ is complete, $B_R(y_0)$ is a compact subset. Furthermore, $F$ is proper and thus $u_i$ attains a maximum at some $(x_i, t_i) \in M\times (0,T]$. From the choice of $(\bar x_i, s_i)$ and $\epsilon_i$ in (\ref{choice of epsilon i}),
$$u(x_i, t_i) \ge u_i(x_i, t_i) \ge u_i(\bar x_i, s_i) \ge u(\bar x_i, s_i) -\frac 1i.$$
Thus we have
$$u(x_i, t_i) \to \sup u.$$
Now we consider the derivatives of $u$ at $(x_i, t_i)$. If $F(x_i, t_i)$ is not in the cut locus of $y_0$, then $\rho$ is differentiable at $(x_i, t_i)$. Then so is $u_i$ and we have
\begin{equation} \label{eqn 1}
\nabla^{M_{t_i}} u_i = 0\ \ \text{ and }\ \left(\ppt{} - \Delta^{M_{t_i}} \right) u_i \ge 0 \ \ \text{ at } (x_i, t_i).
\end{equation}
(The inequality holds since $t_i >0$). The first equality implies
\begin{equation} \label{gradient of rho^2}
\nabla^{M_{t_i}} u = \epsilon_i \nabla^{M_{t_i}} \rho^2 = 2\epsilon_i \rho (\nabla r)^\top 
\end{equation}
at $(x_i, t_i)$, where $(\cdot)^\top$ denotes the projection onto $T_{x_i} M_{t_i}$. Let $\{e_1, \cdots, e_n\}$ be any orthonormal basis at $T_{x_i}M_{t_i}$ with respect to $g_{t_i}$. Then
\begin{equation} \label{Lap of rho ^2}
 \Delta^{M_{t_i}} \rho^2 = 2 \sum_{i=1}^n |\nabla_{e_i} ^{M_{t_i}} r |^2 + 2\rho \sum_{i=1}^n \nabla ^2 r(e_i, e_i) + 2\rho \bar g(\nabla r , \vec H).
\end{equation}
Next we use the lower bound on $(n-1)$-sectional curvature of $\overline M$ to obtain the following lemma.
\begin{lem} \label{lemma in main thm}
There is $C_1 = C_1(n, C)>0$ so that
\begin{equation} \label{consequence of Li wang comparison}
\sum_{i=1}^n  \nabla ^2 r(e_i, e_i) \le C_1\rho.
\end{equation}
\end{lem}
\begin{proof}[Proof of lemma]: We consider two cases. First, if $\gamma'$ is perpendicular to $T_{x_i}M_{t_i}$, write
\begin{equation*}
\sum_{i=1}^n \nabla ^2 r(e_i, e_i) = \frac{1}{n-1} \sum_{j=1}^n \sum_{i\neq j} \nabla ^2 r(e_i, e_i).
\end{equation*}
Since $\overline M$ has $(n-1)$-sectional curvature bounded from below by $-C$, we apply Theorem \ref{comp thm of LW} to the plane $V$ spanned by $\{e_1, \cdots , e_n\} \setminus \{e_i\}$ for each $i$. Thus
\begin{equation} \label{lemma in main thm first case}
\begin{split}
\sum_{i=1}^n \nabla ^2 r(e_i, e_i) &\le \frac{n}{n-1} \sum_{j=1}^{n-1}  \sqrt C \rho \coth (\sqrt C \rho) \\
&= n  \sqrt C \rho \coth (\sqrt C \rho) .\\
\end{split}
\end{equation}

Second, if $\gamma'$ is not perpendicular to $T_{x_i} M_{t_i}$, since the right hand side of (\ref{consequence of Li wang comparison}) is independent of the orthonormal basis chosen, we can assume that $e_1$ is parallel to the projection of $\gamma'$ onto $T_{x_i}M_{t_i}$. Write
\begin{equation*}
e_1 = e_1^\perp + a\gamma',
\end{equation*}
where $e_1^\perp$ lies in the orthogonal complement of $\gamma'$ and $a= \langle e_1, \gamma'\rangle$. By a direct calculation,
\begin{equation} \label{Hr = Hr perp}
\begin{split}
\nabla^2 r(e_1, e_1) &=(\nabla_{e_1} \nabla r)(e_1) \\
&=e_1 \langle \gamma' , e_1\rangle - \langle \gamma' , \nabla_{e_1}e_1\rangle \\
&= \langle \nabla_{e_1} \gamma', e_1\rangle \\
&= \langle \nabla_{e_1^\perp + a\gamma' } \gamma' , e_1^\perp + a\gamma'\rangle \\
&= \langle \nabla_{e_1^\perp} \gamma' , e_1^\perp \rangle + a \langle \nabla_{e_1^\perp} \gamma', \gamma' \rangle \\
&= \nabla^2 r (e_1^\perp, e_1^\perp).
\end{split}
\end{equation}
We further split into two situations. If $e_1^\perp = 0$, then the above shows $\nabla^2 r(e_1, e_1) = 0$. Using Theorem \ref{comp thm of LW} we conclude
\begin{equation} \label{lemma in main thm second case}
\begin{split}
\sum_{i=1}^n \nabla ^2 r(e_i, e_i) & =\sum_{i=2}^n \nabla ^2 r(e_i, e_i) \\
&\le (n-1)   \sqrt C \rho \coth (\sqrt C \rho) \\
\end{split}
\end{equation}

If $e_1^\perp \neq 0$, write $b = \| e_1^\perp \|$ and $f_1 = b^{-1} e_1^\perp$. Then $\{f_1, e_2, \cdots, e_n\}$ is an orthonormal basis of a $n$-dimensional plane in $T_{F(x_i,t_i)}\overline M$ orthogonal to $\gamma'$. Using (\ref{Hr = Hr perp}),
\begin{equation*}
\begin{split}
\sum_{i=1}^n \nabla ^2 r(e_i, e_i) &=  \nabla ^2 r(e_1^\perp, e_1^\perp) +\sum_{i=2}^n \nabla ^2 r(e_i, e_i) \\
&= b^2 \nabla^2 r(f_1, f_1) + \sum_{i=2}^n \nabla ^2 r(e_i, e_i) \\
&= b^2\left( \nabla^2 r(f_1, f_1) + \sum_{i=2}^n \nabla ^2 r(e_i, e_i)\right) + (1-b^2)  \sum_{i=2}^n \nabla ^2 r(e_i, e_i).
\end{split}
\end{equation*}
Now we apply Theorem \ref{comp thm of LW} again (note that the first term can be dealt with as in (\ref{lemma in main thm first case}))
\begin{equation} \label{lemma in main thm third case}
\begin{split}
\sum_{i=1}^n \nabla ^2 r(e_i, e_i) &\le b^2 n\sqrt C \rho \coth (\sqrt C \rho) + (1-b^2) (n-1)   \sqrt C \rho \coth (\sqrt C \rho) \\
&\le n\sqrt C \rho \coth (\sqrt C \rho). 
\end{split}
\end{equation}
Summarizing (\ref{lemma in main thm first case}), (\ref{lemma in main thm second case}) and (\ref{lemma in main thm third case}), we have
\begin{equation*}
\sum_{i=1}^n \nabla^2 r(e_i, e_i) \le n\sqrt C \rho \coth (\sqrt C \rho) \le C_1 \rho
\end{equation*}
for some $C_1 = C_1(n, C)>0$. Thus the lemma is proved.
\end{proof}

Using Lemma \ref{lemma in main thm}, (\ref{Lap of rho ^2}) and $\ppt \rho^2 = 2\rho \bar g(\nabla r , \vec H)$,
\begin{equation} \label{box of rho ^2}
\begin{split}
\left(\ppt{} - \Delta^{M_{t_i}}\right) \rho^2 &= -2 \sum_{i=1}^n |\nabla_{e_i} ^{M_{t_i}} r |^2 - 2\rho \sum_{i=1}^n \nabla ^2 r(e_i, e_i) \\
&\ge - 2n - 2C_1 \rho\\
\end{split}
\end{equation}
(\ref{gradient of rho^2}) and (\ref{box of rho ^2}) imply that at $(x_i, t_i)$ we have respectively
\begin{equation} \label{inequality on grad u}
|\nabla u|\le 2\epsilon_i \rho
\end{equation}
and
\begin{equation} \label{inequality on box u}
\left( \ppt{} - \Delta^{M_{t_i}}\right) u \ge -2\epsilon_i (n+C_1 \rho).
\end{equation}
Note
$$ u(x_i, t_i) -\epsilon_i \rho(x_i, t_i)^2 = u_i (x_i, t_i)   \ge u_i(y,s)>0.$$
This implies
$$ \rho(x_i,t_i)^2 \le u(x_i,t_i)\epsilon_i^{-1}.$$
Using the sub-linear growth condition (3) of $u$ and Young's inequality, we have 
\begin{equation*}
\begin{split}
 \rho(x_i, t_i)^2 &\le B\epsilon_i^{-1} + B \epsilon_i^{-1} \rho (x_i, t_i)^\alpha\\
&\le B \epsilon_i^{-1} + \frac 12 \rho(x_i, t_i)^2 + \frac 12 (B\epsilon_i^{-1})^{\frac{2}{2-\alpha}}.
\end{split}
\end{equation*}
Thus we get
$$ \rho (x_i, t_i)\epsilon_i \le \sqrt{2B}\sqrt{\epsilon_i}+B^{\frac{1}{2-\alpha}}\epsilon_i^{\frac{1-\alpha}{2-\alpha}} . $$
Together with (\ref{inequality on grad u}), (\ref{inequality on box u}) and that $\epsilon_i \to 0$,
\begin{equation*}
|\nabla u|(x_i, t_i) \to 0, \ \ \ \liminf_{i\to \infty} \left(\ppt{} -\Delta^{M_{t_i}}\right) u(x_i, t_i) \ge 0.
\end{equation*}
This proves the theorem if $\rho$ is smooth at $(x_i, t_i)$ for all $i$. When $\rho$ is not differentiable at some $(x_i, t_i)$, one applies the Calabi's trick by considering $r_\epsilon  (y) = d_{\bar g} (y, y_\epsilon)$ instead of $r$, where $y_\epsilon$ is a point closed to $y_0$. The method is standard and thus is skipped.
\end{proof}

\begin{rem}
Condition (1) in the above theorem is used solely to exclude the case that $u_i$ is maximized at $(x_i, 0)$ for some $x_i \in M$. The condition can be dropped if that does not happen (see the proof of Theorem \ref{OY for self shrinker}).
\end{rem}


\section{Preservation of Gauss image}
In this section we assume that $F_0 : M^n\to \mathbb R^{n+m}$ is a proper immersion. Let $F : M\times [0,T] \to \mathbb R^{n+m}$ be a mean curvature flow starting at $F_0$. We further assume that the second fundamental form are uniformly bounded: there is $C_0>0$ so that
\begin{equation} \label{A bounded}
\| A(x,t)\| \le C_0,\ \ \ \text{for all } (x, t) \in M\times [0,T].
\end{equation}

\begin{lem} \label{F is proper}
The mapping $F$ is proper.
\end{lem}

\begin{proof}
Let $B_0(r)$ be the closed ball in $\mathbb R^{n+m}$ centered at the origin with radius $r$. Then by (\ref{MCF eqn}) and (\ref{A bounded}) we have
\begin{equation*}
\begin{split}
|F(x, t) - F(x,0)| &= \left| \int_0^t\frac{\partial F}{\partial s} (x, s)ds\right| \\
&= \left| \int_0^t \vec H(x,s) ds \right| \\
&\le \sqrt n\int_0^t\| A(x, s)\| ds\\
&\le  C_0 \sqrt nT.
\end{split}
\end{equation*}
Thus if $(x, t) \in F^{-1}(B_0(r))$, then $x$ is in $F_0^{-1}(B_0(r + C_0\sqrt nT))$. Let $(x_n, t_n) \in F^{-1}(B_0(r))$. Since $F_0$ is proper, a subsequence of $\{x_n\}$ converges to $x\in M$. Since $[0,T]$ is compact, a subsequence of $(x_n, t_n)$ converges to $(x, t)$, which must be in $F^{-1}(B_0(r))$ since $F$ is continuous. As $r>0$ is arbitrary, $F$ is proper.
\end{proof}

In particular, the parabolic Omori-Yau maximum principle (Theorem \ref{parabolic OY}) can be applied in this case.

Let $G(n,m)$ be the real Grassmanians of $n$-planes in $\mathbb R^{n+m}$ and let
\begin{equation}\label{dfn of gauss map}
\gamma : M \times [0,T]\to G(n,m), \ \ \ x\mapsto F_*T_xM
\end{equation}
be the Gauss map of $ F$.

Now we prove Theorem \ref{main theorem}, which is a generalization of a Theorem of Wang \cite{Wang} to the noncompact situation with bounded second fundamental form.

\begin{proof} [Proof of Theorem \ref{main theorem}]
Let $d : G(n,m) \to \mathbb R$ be the distance to $\Sigma$. That is $d(\ell ) =  \inf_{L \in \Sigma} d(L, \ell)$. Since $\gamma (\cdot,0) \subset \Sigma$, we have $d\circ \gamma = 0$ when $t=0$. Using chain rule and (\ref{A bounded}), as $d\gamma = A$,
$$ d(\gamma(x,t)) = d(\gamma(x,t)) - d(\gamma(x,0)) = \int_0^t \nabla d \circ d\gamma (x, s) ds \le tC_0 .$$
Since $\Sigma \subset G(n,m)$ is compact, there is $\epsilon_0 >0$ so that the open set
$$ V  = \{ \ell \in G(n,m): d(\ell, \Sigma) <\sqrt\epsilon_0\}$$
lies in a small tubular neighborhood of $\Sigma$ and the function $d^2$ is smooth on this neighborhood.  Let $T' = \epsilon_0/2C_0$. Then the image of $f := d^2 \circ \gamma$ lies in this tubular neighborhood if $t\in [0,T']$ and $f$ is a smooth function on $M \times [0,T']$.

The calculation in Wang \cite{Wang} shows that
\begin{equation}
\left(\ppt {} -\Delta\right) f \le C |A_t|^2 f,
\end{equation}
where $C>0$ depends on $\epsilon_0$ and $\Sigma$. Together with (\ref{A bounded}) this shows that
$$\left(\ppt {} -\Delta\right) f \le C_1 f$$
for some positive constant $C_1$.

Let $g = e^{-(C_1+1) t} f$. Then $g$ is bounded, nonnegative and $g(\cdot, 0)\equiv 0$.  On the other hand,
\begin{equation} \label{box of g}
\left(\ppt {}-\Delta\right) g = -(C_1+1) g + e^{-(C_1+1)t} \left(\ppt {} -\Delta \right) f \le -g.
\end{equation}
If $g$ is positive at some point, Theorem \ref{parabolic OY} implies the existence of a sequence $(x_i, t_i)$ so that
$$ g(x_i, t_i) \to  \sup g , \ \ \ \limsup_{i\to \infty} \left(\ppt{} - \Delta\right) g(x_i, t_i) \ge 0. $$
Take $i\to \infty$ in (\ref{box of g}) gives $0 \le -\sup g$, which contradicts that $g$ is positive somewhere. Thus $g$ and so $f$ is identically zero. This is the same as saying that $\gamma (x, t) \in \Sigma$ for all $(x, t)\in [0,T']$. Note that $T'$ depends only on $C_0$, so we can repeat the same argument finitely many time to conclude that $\gamma(x, t) \in \Sigma$ for all $(x, t) \in M\times [0,T]$.
\end{proof}

\begin{proof}[Proof of Corollary \ref{preserve Lagrangian}]
An immersion is Lagrangian if and only if its Gauss map has image in the Lagrangian Grassmanians $LG(n)$, which is a totally geodesic submanifold of $G(n, n)$. The Corollary follows immediately from Theorem \ref{main theorem}.
\end{proof}

\begin{rem} \label{Neves, Sm}
Various forms of Corollary \ref{preserve Lagrangian} are known to the experts. In \cite{N}, the author comments that the argument used in \cite{Sm} can be generalized to the complete noncompact case, if one assumes the following volume growth condition:
$$\text{Vol} (L_0 \cap B_R(0)) \le C_0 R^n,\ \ \ \text{for some } C_0 >0.$$
The above condition is needed to apply the non-compact maximum principle in \cite{EH}.
\end{rem}


\section{Omori-Yau maximum principle for self-shrinkers}
In this section, we improve Theorem 5 in \cite{CJQ} using Theorem \ref{parabolic OY}. The proof is more intuitive in the sense that we use essentially the fact that a self-shrinker is a self-similar solution to the mean curvature flow (possibly after reparametrization on each time slice).

First we recall some facts about self-shrinker. A self-shrinker to the mean curvature flow is an immersion $\widetilde F : M^n \to \mathbb R^{n+m}$ which satisfies
\begin{equation} \label{ss eqn}
\widetilde F^\perp = -\frac 12 \vec H.
\end{equation}
Fix $T_0 \in (-1,0)$. Let $\phi_t : M\to M$ be a family of diffeomorphisms on $M$ so that
\begin{equation}
\phi_{T_0} = \text{Id}_M,\ \ \ppt{} \big( \widetilde F(\phi_t (x))\big) =\frac{1}{2(-t)} \widetilde F^\top(\phi_t(x)),\ \ \ \forall t\in [-1, T_0].
\end{equation}
Let
\begin{equation} \label{MCF defined by ss}
F(x, t) = \sqrt{-t} \widetilde F(\phi_t(x)),\ \ \  (x, t) \in M\times [-1,T_0].
\end{equation}
Then $F$ satisfies the MCF equation since by (\ref{ss eqn}),
\begin{equation*}
\begin{split}
\ppt{F} (x,t) &= \ppt{} \big( \sqrt{-t} \widetilde F (\phi_t(x))\big)  \\
&= -\frac{1}{2\sqrt {-t}} \widetilde F(\phi_t(x) ) + \sqrt{-t} \ppt{} \big( \widetilde F(\phi_t (x))\big) \\
&= -\frac{1}{2\sqrt {-t}} \widetilde F(\phi_t(x) ) +\frac{1}{2\sqrt {-t}} \widetilde F^\top(\phi_t(x) )  \\
&= \frac{1}{\sqrt{-t}} \vec H_{\widetilde F} (\phi_t(x)) \\
&= \vec H_F (x,t).
\end{split}
\end{equation*}
Lastly, recall the $\mathcal L$ operator defined in \cite{CM}:
\begin{equation}
 \mathcal L f = \Delta f - \frac 12 \langle \nabla f, \widetilde F^\top\rangle.
\end{equation}

We are now ready to prove Theorem \ref{OY for self shrinker}:

\begin{proof} [Proof of Theorem \ref{OY for self shrinker}] Recall $T_0 \in (-1,0)$. Let $u : M\times [-1,T_0] \to \mathbb R$ be given by
\begin{equation}
u(x,t) = f(\phi_t(x)),\ \ \ \forall (x, t) \in M \times [-1, T_0].
\end{equation}
Then
$$u(x,t) \le C(1+ |\widetilde F(\phi_t(x)|^\alpha) \le C(-T_0)^{-\alpha/2} |F(x, t)|^\alpha.$$
Thus we can apply Theorem \ref{parabolic OY} (The condition that $u(\cdot, 0)\equiv 0$ in Theorem \ref{parabolic OY} is used only to exclude the case $t_i = -1$. But since
$$ u_i(x,t) = f(\phi_t(x)) - \epsilon_i | \sqrt{-t} \widetilde F(\phi_t(x))|^2,$$
in order that $u_i$ is maximized at $(x_i, t_i)$ we must have $t_i = T_0$. In particular $t_i \neq -1$). Thus there is a sequence $(x_i, T_0)$ so that
$$  u(x_i, T_0) \to \sup u ,\ \ |\nabla^{M_{T_0}} u (x_i, T_0)| \to 0,\ \ \ \liminf_{i\to \infty}\left(\ppt{} - \Delta^{M_{T_0}}\right) u (x_i, T_0) \ge 0 .$$
Using $\phi_{T_0} = \text{Id}$ and the definition of $u$, the first condition gives
\begin{equation}
f(x_i) \to \sup f.
\end{equation}
Since $\nabla^{M_{T_0}} = \frac{1}{\sqrt{-T_0}} \nabla ^{M}$, the second condition gives
\begin{equation}
|\nabla^{M}f (x_i)| \to 0.
\end{equation}
Lastly,
\begin{equation}
\ppt u (x_i,T_0) = \ppt f (\phi_t(x))\bigg|_{t=T_0} = \frac{1}{2(-T_0)} \langle \nabla f(x_i), \widetilde F^\top(x_i)\rangle
\end{equation}
and
\begin{equation*}
\Delta^{M_{T_0}} u(x_i, T_0)= \Delta^{M_{T_0}} f(x_i) = \frac{1}{-T_0} \Delta^M f(x_i).
\end{equation*}
Thus
\begin{equation*}
\left(\ppt{} - \Delta^{M_{T_0}}\right) u (x_i, T_0)  = \frac{1}{T_0} \mathcal L f(x_i)
\end{equation*}
and the result follows.
\end{proof}

\begin{rem}
Note that the above theorem is stronger than Theorem 5 in \cite{CJQ}, where they assume that $f$ is bounded above (which corresponds to our case when $\alpha = 0$).
\end{rem}

\begin{rem}
Our growth condition on $f$ is optimal: the function $f(x)= \sqrt{|x|^2+1}$ defined on $\mathbb R^n$ (as a self-shrinker) has linear growth, but the gradient of $f$
$$\nabla f = \frac{x}{\sqrt{|x|^2+1}}$$
does not tend to $0$ as $f(x)\to \sup f =\infty$.
\end{rem}

\begin{rem}
In Theorem 4 of \cite{CJQ}, the authors also derive a Omori-Yau maximum principle on a properly immersed self-shrinker for the Laplace operator. There they assume $u: M\to \mathbb R$ satisfies the growth condition
$$\lim_{x\to \infty}\frac{u(x)}{\log\left(\sqrt{|\widetilde F(x)|^2+ 4}-1\right)} = 0.$$
We remark that the condition can be weaken to
$$\lim_{x\to \infty} \frac{u(x)}{|\widetilde F(x)|+1} = 0,$$
since the Laplacian of the function $ |\widetilde  F|^2$ satisfies better estimates: $\Delta |\widetilde F|^2 \le 2n$. Thus one can argue as in p.79 in \cite{AMR} to conclude.
\end{rem}

\bibliographystyle{amsplain}

\end{document}